\documentclass[11pt,a4paper]{amsart}
\usepackage{texmac}
\usepackage{enumitem}
\usepackage{mathtools}
\usepackage{hyperref}
\usepackage[all,cmtip]{xy}
\usepackage{color}

\operators{Stab,Prim,Inn,Out,ss,B,G,R,K,Inf,Def,vol,Ext,St,ab,ind,res,GL,nrd,PSL,PGL,Ann}
\operators{Sym,Alt,Int,inn,N,Def,U,transp,Fix,bad,sig,SU,SO}

\bbsymbols{b}{F,Z,Q,P,R,C,H}
\calsymbols{c}{H,U,X,Y,P,M,D,C,F,O,L,T,S}
\fraksymbols{f}{H,l,m,t,M,L}
\let\ge\geqslant

\def\dsum_#1_#2{\sum_{{#1}\atop {#2}}}
\def\rar{\rightarrow}

\newcommand{\good}{$G$-disjoint}
\newcommand{\decent}{closed connected orientable}
\newcommand{\act}{\cdot}

\newcommand{\killed}{\gamma}
\newcommand{\compl}{\gamma'}
\newcommand{\meria}{\cM}

\newcommand{\paraa}{\cL}

\newcommand{\centre}{{\rm Z}}

\newcommand{\bilin}{\phi}

\newcommand{\permchar}[1]{\pi(#1)}
\newcommand{\grpalg}{\Q[G]}
\newcommand{\intgrpalg}{\Z[G]}
\newcommand{\regmod}[1]{{{}_{#1}{#1}}}
\newcommand{\regmodg}{\regmod{\grpalg}}

\newcommand{\Mat}{{\rm M}}

\newcommand{\origM}{M}

\newcommand{\drilledM}[2]{{#1}(#2)}
\newcommand{\filledM}[2]{{#1}(#2)}
\newcommand{\emb}{\varphi}
\newcommand{\dotsth}[1]{f_{#1}}
\newcommand{\dotlambda}{\dotsth{\lambda}}
\newcommand{\dotmu}{\dotsth{\mu}}

\title{Group representations in the homology of $3$-manifolds}
\author{Alex Bartel}
\address{School of Mathematics and Statistics, University of Glasgow,
Glasgow G12 8QQ, Scotland, UK}
\author{Aurel Page}
\address{INRIA, Univ. Bordeaux, CNRS, IMB, UMR5251, F-33400 Talence, France}
\email{alex.bartel@glasgow.ac.uk, aurel.page@inria.fr}
\begin{document}
\maketitle
\begin{abstract}
If $M$ is a manifold with an action of a group $G$, then
the homology group $H_1(M,\Q)$ is naturally a $\grpalg$-module, where
$\grpalg$ denotes the rational group ring. We prove that for every finite
group~$G$, and for every $\grpalg$-module $W$, there exists a closed
hyperbolic~$3$-manifold $M$ with a free $G$-action such that the $\grpalg$-module
$H_1(M,\Q)$ is isomorphic to $W$. We give an application to spectral geometry:
for every
finite set~$\cP$ of prime numbers, there exist hyperbolic~$3$-manifolds
$N$ and $N'$ that are strongly isospectral such that for all $p\in \cP$,
the $p$-power torsion subgroups of $H_1(N,\Z)$ and of $H_1(N',\Z)$ have
different orders.
%, while for all $p\not\in \cP$, the two $p$-power torsion
%subgroups are isomorphic.
%Our techniques also apply to
%non-free actions. In particular,
%We also show that, in a certain precise sense, the rational homology
%of oriented Riemannian $3$-manifolds with a $G$-action ``knows''
%nothing about the fixed point structure under $G$, in contrast to the
%$2$-dimensional case.
%For both the main result and the isospectral application,
%we apply techniques from the theory of rational and integral representations of
%finite groups, some of which are quite recent and were
%originally developed in number theoretic contexts.
The main geometric techniques are Dehn surgery and, for the spectral application,
the Cheeger-M\"uller formula, but we also make use of tools from
different branches of algebra, most notably of regulator
constants, a representation theoretic tool that was
originally developed in the context of elliptic curves.

%We also prove that given a $\grpalg$-module $V$ and
%a $3$-manifold $M$ with a not necessarily free $G$-action, we may produce a $3$-manifold $M'$
%with a $G$-action, satisfying $H_1(M',\Q)\cong H_1(M,\Q)\oplus V$ and having
%the same fixed point structure under $G$ as $M$. We deduce that for oriented
%Riemannian $3$-manifolds with a $G$-action, the rational homology ``knows''
%nothing about the fixed point structure under $G$, in contrast to the
%$2$-dimensional case.
%We also briefly discuss the problem in other
%dimensions. In dimension 2, we show the conclusion of
%the main result to be false, and in dimensions 4 and higher we
%deduce it from basic known results.
\end{abstract}

\section{Introduction}
\subsection{Group actions on rational homology of $3$-manifolds}
If $M$ is a manifold with an action by a group $G$, then the homology
of $M$ carries a natural $G$-action. The $G$-module structure
of integral and rational homology can often be used to deduce information
about the manifold, see e.g. \cite{Edmonds,Sikora}.

In this paper, we investigate the $G$-module structure of the rational
homology of $3$-manifolds.
In \cite{CooperLong} Cooper and Long prove that for every finite group $G$,
there exists a hyperbolic rational homology $3$-sphere with a free $G$-action.
In fact, their method proves a stronger statement. Let $\grpalg$ denote
the group algebra of $G$ over the field $\Q$ of rational numbers. Cooper and Long
define the notion of a canonical $\grpalg$-module, and prove that every direct sum
of canonical $\grpalg$-modules can be realised as $H_1(M,\Q)$ for a closed
$3$-manifold $M$ with a free $G$-action.

%\begin{theorem*}[Cooper and Long \cite{CooperLong}]
%Let $G$ be a finite group, and let $W$ be isomorphic to a finite direct sum
%of canonical modules. Then there exists a \decent~$3$-manifold $M$ with a
%free $G$-action such that the $\grpalg$-module $H_1(M,\Q)$ is isomorphic to~$W$.
%\end{theorem*}

In the present paper, we generalise the theorem of Cooper and Long to
arbitrary $\grpalg$-modules.
%By a \emph{hyperbolic manifold} we mean
%a compact connected oriented smooth manifold whose interior is equipped with
%a Riemannian metric with constant curvature $-1$.
Our main result is the following.

\begin{theorem}\label{thm:intromain}
Let $G$ be a finite group, and let $W$ be a finitely generated $\grpalg$-module.
Then there exists a \decent~$3$-manifold $M$ with a free
orientation preserving $G$-action, such that the $\grpalg$-module $H_1(M,\Q)$ is isomorphic to~$W$.
\end{theorem}

Like Cooper and Long, we can also strengthen Theorem \ref{thm:intromain} by
ensuring that $M$ is hyperbolic -- see Theorem \ref{thm:bodymain}.

\begin{remark}
The referees inform us that there is a way of proving Theorem~\ref{thm:intromain}
by using higher dimensional surgery. The proof we will present here will only
use surgery in dimension $3$.
\end{remark}

\subsection{Application to isospectral manifolds}
As an application of Theorem \ref{thm:intromain} we prove a result on
torsion homology of isospectral manifolds. This relies on our previous work \cite{us}
on the interplay between Sunada's construction of isospectral manifolds \cite{Sunada}
and the Cheeger--M\"uller theorem.

Recall that two Riemannian manifolds
$M$ and $M'$ are said to be \emph{strongly isospectral} if the spectra of
every natural (see
  \cite[Section II, paragraph before Examples 3]{Pesce} for a precise
  definition) self-adjoint elliptic differential
operator on~$M$ and~$M'$ agree.
There is a large body of literature devoted to investigating which topological
or geometric invariants of manifolds are strongly isospectral invariants, see
\cite{Schueth,GordonSchueth,gordon} for surveys. Strongly isospectral Riemannian
manifolds necessarily have the same dimension, the same volume, and the same
Betti numbers, but for example they may have non-isomorphic real cohomology
rings, see \cite{Lauret}. Ikeda has shown that (strongly) isospectral closed
$3$-manifolds with constant positive curvature are necessarily isometric \cite{Ikeda}.
In contrast, we show in the present paper that
strongly isospectral hyperbolic $3$-manifolds can have very different integral
homology groups. If $p$ is a prime number, and $A$ is an Abelian group, let
$A[p^\infty]$ denote the subgroup of $A$ of elements of $p$-power order.

\begin{theorem}\label{thm:introsiso}
Let $\cP$ be a finite set of prime numbers. Then
there exist \decent~$3$-manifolds $M$ and $M'$ that are strongly isospectral with
respect to hyperbolic metrics and such that
\begin{enumerate}[leftmargin=*]
\item for all $p\in \cP$ we have
$$
\#H_1(M,\Z)[p^\infty]\neq \#H_1(M',\Z)[p^\infty];
$$
\item for all prime numbers $q\not\in \cP$ we have an isomorphism of Abelian groups
$$
H_1(M,\Z)[q^\infty]\cong H_1(M',\Z)[q^\infty].
$$
\end{enumerate}
\end{theorem}

\smallskip

\begin{remarks*}~
\begin{enumerate}[leftmargin=*]
\item We will obtain the manifolds $M$ and $M'$ in Theorem \ref{thm:introsiso}
using a construction of Sunada \cite{Sunada}, who was guided by a well-known
analogy between spectral zeta functions of manifolds and Dedekind zeta functions
of number fields.
\item 
A weak form of the number theoretic analogue of Theorem \ref{thm:introsiso} is
an old open problem \cite{deSmit1, deSmit2}: do there exist, for every prime
number $p$, number fields with the same Dedekind zeta function but with
different $p$-class numbers?
\item
Theorem \ref{thm:introsiso} certainly does not hold for $2$-manifolds, since
they have torsion-free homology, while for $4$- and higher-dimensional manifolds,
Theorem \ref{thm:introsiso} has already been known since the work of Sunada.
The $3$-dimensional case has been the only open one.

%\item The link between Theorem \ref{thm:intromain} and isospectral manifolds
%is Sunada's construction, which produces isospectral manifolds out of certain
%$G$-coverings of Riemannian manifolds. See Section \ref{sec:isospectral} for the
%details.
\end{enumerate}
\end{remarks*}

Let $p$ be an odd prime number. Let us briefly explain how to deduce Theorem
\ref{thm:introsiso} in the special case $\cP=\{p\}$ from Theorem \ref{thm:intromain}.
Consider the following two subgroups of the group $G=\GL_2(\F_p)$ of invertible
$2\times 2$ matrices over the finite field with $p$ elements:
$$
B=\begin{pmatrix}\F_p^\times &\F_p\\0& \F_p^\times\end{pmatrix},\;\;\;
H=\begin{pmatrix}(\F_p^\times)^2 &\F_p\\0& \F_p^\times\end{pmatrix}.
$$
We can form the permutation modules $\Q[G/H]$ and $\Q[G/B]$, which are spanned
over $\Q$ by the cosets of the respective subgroup, and where $G$ acts by
permuting the respective basis. The module $\Q[G/H]$ decomposes as a direct sum
$\Q[G/H]\cong \Q[G/B]\oplus I$,
where $I$ is a simple $\grpalg$-module of dimension $(p+1)$ over $\Q$.
The first part of Theorem \ref{thm:introsiso} for $\cP=\{p\}$
immediately follows from Theorem \ref{thm:intromain} and the following result.

\begin{lemma}[\cite{us}, Corollary~4.4]\label{lem:corus}
Let $p$ be an odd prime number, let $G=\GL_2(\F_p)$, and let $I$ be as above.
Suppose that there exists a closed $3$-manifold $X$ with a free
$G$-action, such that the multiplicity of $I$ in the $\grpalg$-module
$H_1(X,\Q)$ is odd. Then there exist \decent~Riemannian $3$-manifolds $M$ and $M'$
that are strongly isospectral and such that
$$
\#H_1(M,\Z)[p^\infty]\neq \#H_1(M',\Z)[p^\infty].
$$
If, in addition, $X$ is
hyperbolic, then $M$ and $M'$ can be chosen to be hyperbolic.
\end{lemma}
%\noindent
By inspecting the construction in Lemma \ref{lem:corus}
a bit more closely, one can also deduce the second part of Theorem
\ref{thm:introsiso} from \cite[Theorem 3.5]{us}.

For concrete groups $G$ and $\Q[G]$-modules $W$, one can sometimes try to reach
the conclusion of Theorem \ref{thm:intromain} by a brute force search. In
\cite[Proposition 1.5]{us}, we were able to prove Theorem
\ref{thm:introsiso}(1) in this way when $\cP=\{p\}$ for all $p\leq 71$.

\begin{remark*}
The canonicity condition on $W$ in the construction of Cooper and Long 
can be formulated as follows. Let $\regmodg$ denote the left regular $\grpalg$-module.
Recall that every $\grpalg$-module can be uniquely written as a direct sum of
simple modules. The condition on $W$ for the method of \cite{CooperLong} to apply
is that
for every simple $\grpalg$-module $V_i$, the multiplicity of $V_i$ in $W$
be divisible by the multiplicity of $V_i$ in the regular module $\regmodg$.
Note that the multiplicity of $I$ in the regular module of $\GL_2(\F_p)$ is $p+1$,
so if $W$ is a canonical $\Q[\GL_2(\F_p)]$-module, then
the multiplicity of $I$ in $W$ is even. In particular, the result of Cooper--Long
is not sufficient to apply Lemma~\ref{lem:corus}.
\end{remark*}

The proof of the general case of Theorem \ref{thm:introsiso} will involve the
same ideas as that of the special case sketched above, but will require more
algebraic preparation, and will occupy Section \ref{sec:isospectral}.

\subsection{Ingredients of the proof}
The proof of the main theorem will be given in Section \ref{sec:geom}.
We will show how, given a $3$-manifold $M$ with a free $G$-action, and a
finitely generated $\grpalg$-module $V$, one may perform a sequence of
$G$-equivariant Dehn surgeries to produce a $3$-manifold $M'$ with a free $G$-action
such that there is an isomorphism of $\grpalg$-modules
$H_1(M',\Q)\cong H_1(M,\Q)\oplus V$ -- see Corollary \ref{cor:add};
and also how, given a $3$-manifold $M$ with a free $G$-action such that
$H_1(M,\Q)$ has a $\grpalg$-submodule isomorphic to $\regmodg$, one can ``kill''
that submodule -- see Proposition \ref{prop:removeeQG}. Starting from a $3$-manifold
with a free $G$-action and arbitrary homology, one can then iterate the above
two steps to realise any given $\grpalg$-module -- see Theorem~\ref{thm:bodymain}.

One of these surgeries is prescribed by the coefficients of an idempotent
$e\in \grpalg$ such that $\grpalg e\cong V$, and for this step to
yield the desired result, we need $e$ to satisfy a technical algebraic
condition. The following result, which we will prove in Section
\ref{sec:alg} as Corollary \ref{cor:goodidempotent}, says that all
idempotents in $\grpalg$ indeed do possess the required property.

\begin{proposition}\label{prop:goodidempotent}
Let $G$ be a finite group.
Given an element $x=\sum_{g\in G}a_gg$ of $\grpalg$, where $a_g\in \Q$,
define $x^*=\sum_{g\in G}a_gg^{-1}\in \grpalg$.
Then for every idempotent $e\in \grpalg$, we have
$\regmodg=\grpalg e + \grpalg (1-e^*)$.
\end{proposition}
Note that if the star is dropped, then the conclusion immediately
follows from the definition of an idempotent. On the other hand, since $e$
is not assumed to be central, one does not, in general, have $e=e^*$.
%Also,
%for every idempotent $e\in \grpalg$, one has an isomorphism of $\grpalg$-modules
%$\grpalg(1-e^*)\cong \grpalg(1-e)$, and hence an abstract isomorphism
%$\regmodg\cong \grpalg e\oplus \grpalg (1-e^*)$.
%But Proposition~\ref{prop:goodidempotent}, whose conclusion is an equality,
%rather than an abstract isomorphism, is actually quite subtle. For example,
%not every idempotent $e\in \grpalg$ for which $\grpalg e\cong V$
%satisfies the conclusions of the proposition. 
Moreover, if the operator $x\mapsto x^*$ is replaced by a different involution
(see Section~\ref{sec:semisimple}) on~$\grpalg$, then there may not exist an
idempotent with the required properties at all. It is important to note that
Proposition \ref{prop:goodidempotent} does not follow from the well-known fact
that $\grpalg e^*$ is abstractly isomorphic to $\grpalg e$; see in particular
Example \ref{ex:nonisotropic}.

\subsection{Generalisations}
The main geometric step, in which we add a given $\grpalg$-module
to the homology of a given $3$-manifold with a free $G$-action actually
works in greater generality. For example, instead of a free $G$-action we may
allow an orientation preserving action by isometries with no element acting trivially,
which implies that the fixed point set under every $g\in G$ is at most
$1$-dimensional. For the precise statements, see Theorems \ref{thm:generaladd}
and~\ref{thm:general}.

This has the somewhat surprising consequence that given a \decent~Riemannian
$3$-manifold $M$ with an orientation preserving effective $G$-action,
one can infer no information about the fixed points from the structure
of~$H_1(M,\Q)$: if a certain configuration of fixed point spaces can be
realised at all, then it can be realised with $H_1(M,\Q)$ being isomorphic
to any given $\grpalg$-module. This stands in stark contrast to the
situation in dimension~$2$, as we point out in Section \ref{sec:otherdim}. See
in particular Corollary~\ref{cor:dim2}, which, vaguely speaking, says that
for \decent~surfaces~$M$ with a $G$-action, $H_1(M,\Q)$ ``knows''
everything about the fixed point structure.

%We also briefly discuss in Section \ref{sec:otherdim} the case of $d$-manifolds
%when $d>3$. Recall that for every $d>3$, every finitely
%presented group can be realised as the fundamental group of a $d$-manifold.
%This almost immediately implies that for every dimension $d>3$, a statement
%analogous to Theorem \ref{thm:intromain} is true even for the integral
%homology, and even for finitely presented groups~$G$, see Theorem \ref{thm:dim4}.

In this paper, when we say ``manifold'', we will always mean a closed connected
oriented smooth manifold, all automorphisms will be orientation preserving, and
all maps between manifolds will be smooth.

\begin{acknowledgements}
We are grateful to Nathan Dunfield for making us aware of the paper of Cooper
and Long, and for his incredibly prompt and helpful responses to our questions,
to Allan Edmonds for many helpful remarks concerning the state of knowledge on
the subject and the exposition, to John Smillie and Karen Vogtmann for helpful
discussions, and to Hendrik Lenstra for giving the proof of Proposition
\ref{prop:goodidempotent} that we present here, which is much simpler than
the proof we gave in an earlier draft. We are also very grateful to an anonymous
referee for helpful comments that greatly improved the exposition!
Part of this work was done while the authors
were at the University of Warwick. We thank the Mathematics department at Warwick
for a very stimulating working environment.
Finally, we thank the EPSRC for financial
support via a First Grant, respectively the EPSRC Programme Grant EP/K034383/1
LMF: L-Functions and Modular Forms.
\end{acknowledgements}

\section{Algebras with an involution}\label{sec:alg}
In this section, we will prove Proposition \ref{prop:goodidempotent}.

\subsection{Semisimple algebras}\label{sec:semisimple}
The main reference for this subsection is \cite{Lam}.
All our rings are associative, and have a unit element, denoted by~$1$.
All our modules are left modules, and are assumed to be finitely generated.
If $R$ is a ring, then $\centre(R)$ denotes the centre of~$R$;
the regular module $\regmod{R}$ is defined as having
the same underlying additive group as $R$, and the $R$-action being given by
left multiplication.

Let $K$ be a field. A \emph{$K$-algebra} %or an \emph{algebra over $K$}
is a ring $A$ that is equipped with
a ring homomorphism~$K\rar \centre(A)$. All our $K$-algebras are
finite-dimensional over~$K$. 
If~$A$ is a $K$-algebra, then
the \emph{trace} $\tr_{A/K}(a)$ of an element~$a\in A$
is defined to be the trace of the endomorphism of the~$K$-vector space~$A$ given
by multiplication by~$a$ on the left.

\begin{example}\label{ex:trace}
  Let $A=\grpalg$ be the group algebra of $G$ over $\Q$, and let $a=\sum_{g\in G}a_gg$
  be an arbitrary element of $A$. Then $\tr_{A/\Q}a=(\dim_{\Q}A)\cdot a_1=\#G\cdot a_1$.
\end{example}
If $A$ is a $K$-algebra, then an $A$-module~$V$ is called \emph{simple}
if it has exactly two submodules, $0$ and~$V$; a simple submodule of~$\regmod{A}$
is the same thing as a minimal left ideal of~$A$.
The \emph{Jacobson radical} of a $K$-algebra~$A$ is the set of elements~$a\in A$
that  annihilate every simple $A$-module; it is a two-sided ideal of~$A$.
A $K$-algebra~$A$ is called \emph{semisimple} if its Jacobson radical is~$0$.
For every integer~$n\ge 1$, let~$\Mat_n(K)$ be the $K$-algebra of
$n\times n$ matrices over~$K$. We will use the following basic result.

\begin{lem}\label{lem:ssmat}
Let $K$ be a field and let $A$ be a semisimple $K$-algebra. Then there exists a
finite field extension $L/K$ such that $L\otimes_K A$ is isomorphic to a product
of algebras of the form~$\Mat_n(L)$ for integers~$n\ge 1$.
\end{lem}

An \emph{idempotent} in an algebra~$A$ is an element~$e\in A$ such that~$e^2=e$.
If $e\in A$ is an idempotent, then so is $1-e$, and in this case one has
a decomposition into left ideals $A= Ae\oplus A(1-e)$.
If an algebra~$A$ is semisimple, then every simple~$A$-module is isomorphic to some minimal
left ideal of~$A$, every $A$-module is a direct sum of simple
submodules, and for every left ideal~$I$ in~$A$, there exists an idempotent~$e$
in~$A$ such that~$I = Ae$.

An \emph{anti-automorphism}
of an algebra $A$ is a $K$-linear automorphism $\alpha\colon A\rar A$ 
such that~$\alpha(1)=1$ and~$\alpha(xy)=\alpha(y)\alpha(x)$ for
all $x$, $y\in A$.
An \emph{involution} on $A$ is an anti-automorphism $\iota$ such that
$\iota\circ \iota=\id$.

Let~$V$ be a finite-dimensional vector space over~$K$, equipped with a symmetric
bilinear form~$\bilin: V\times V\to K$. If~$X\subset V$ is a subset, then its orthogonal
complement is defined to be
\[
  X^\perp = \{v\in V\mid \bilin(v,x)=0\text{ for all }x\in X\}.
\]
The bilinear form $\bilin$ is called \emph{non-degenerate} if~$V^\perp=0$, and
it is called
\emph{anisotropic} if for every nonzero~$v\in V$ we have~$\bilin(v,v)\neq 0$.
Note that $\bilin$ is non-degenerate if and only if the induced map
$V\to \Hom(V,K)$ given by $v\mapsto (w\mapsto \bilin(v,w))$ is an isomorphism.
It follows that
if~$\bilin$ is non-degenerate, then for every subspace~$W\subset V$, we
have, by the rank-nullity formula,~$\dim V = \dim W + \dim W^\perp$.
If~$\bilin$ is anisotropic, then it is non-degenerate, and for every
subspace~$W\subset V$ we have~$V = W+W^\perp$.

%\begin{lem}
%Let $K$ be a field and let $A$ be a $K$-algebra.
%If the bilinear trace form
%\begin{eqnarray*}
%  A\times A & \rar & K\\
%  (x,y) & \mapsto & \Tr_{A/K}(xy)
%\end{eqnarray*}
%is non-degenerate, then $A$ is semisimple.
%\end{lem}

%\subsection{Involutions}\label{sec:involutions}
%The reference for this subsection is \cite{invol}.

\begin{lem}\label{lem:sym}
  Let~$A$ be a semisimple $K$-algebra, and let $\iota$ be an involution
  on~$A$. Then for all~$x\in A$ we have~$\tr_{A/K}(x) = \tr_{A/K}(\iota(x))$. In
  particular, the $K$-bilinear form~$(x,y)\mapsto \tr_{A/K}(x\iota(y))$ on~$A$ is symmetric.
\end{lem}
\begin{proof}
  See \cite[13.1 (iv)]{Scharlau}.
\end{proof}

\begin{remark}
  In Lemma~\ref{lem:sym} the semisimplicity assumption is necessary: let~$A$ be
  the $K$-algebra of upper-triangular $2\times 2$ matrices with coefficients
  in~$K$, equipped with the involution
  \[
  \iota\colon\begin{pmatrix}a & b \\ 0 & d\end{pmatrix}
    \mapsto \begin{pmatrix}d & -b \\ 0 & a\end{pmatrix}.
  \]
  Then~$\tr_{A/K}\begin{pmatrix}a & b \\ 0 & d\end{pmatrix} = 2a+d$, which is
  not preserved by~$\iota$.
\end{remark}

Let~$A$ be a semisimple $K$-algebra and~$\iota$ be an involution on~$A$. The
\emph{associated symmetric bilinear form} on~$A$ is
\[
  \bilin_\iota: (x,y)\mapsto \tr_{A/K}(x\iota(y)).
\]
We say that~$\iota$ is \emph{non-degenerate} (resp. \emph{anisotropic})
if~$\bilin_\iota$ is non-degenerate (resp. anisotropic).

%%%%%%%%%%%%%%%%
%%%%%%%%%%%%%%%%
%%%%%%%%%%%%%%%%
%
% TO MODIFY/REMOVE :

\subsection{Idempotents and anisotropic involutions}\label{sec:goodidempotent}

In this subsection we prove the main algebraic result,
Proposition~\ref{prop:involidem}. The proof we give here was communicated
to us by Hendrik Lenstra, and is much simpler than the proof we gave in
an earlier draft of the paper.

\begin{lem}\label{lem:leftrightrk}
  Let~$A$ be a semisimple $K$-algebra. Then for all~$x\in A$ we have $\dim_K Ax =
  \dim_K xA$.
\end{lem}
\begin{proof}
  The result is true if~$A$ is a product of matrix algebras over $K$. Let $A$ be
  an arbitrary semisimple $K$-algebra.
  If $L/K$ is a finite field extension, then we have $\dim_L (L\otimes_K A)x = \dim_K Ax$,
  and similarly for $xA$. The general case of the lemma therefore follows from
  the special case and Lemma~\ref{lem:ssmat}.
\end{proof}

Note that if $A$ is a $K$-algebra with an involution $\iota$, and
$e\in A$ is an idempotent, then $\iota(e)$ is also an idempotent.

\begin{lem}\label{lem:iotaperp}
  Let~$A$ be a semisimple $K$-algebra with a non-degenerate involution~$\iota$. Then for
  every idempotent~$e\in A$ we have~$(Ae)^\perp = A(1-\iota(e))$, where the
  orthogonal complement is taken with respect to~$\bilin_\iota$.
\end{lem}
\begin{proof}
  Since $e$ is idempotent, we have~$A(1-\iota(e))\subset (Ae)^\perp$. On the
  other hand we have
  \begin{eqnarray*}
    \dim (Ae)^\perp 
      &=& \dim A - \dim Ae\\
      &=& \dim A(1-e)\\
      &=& \dim (1-e)A\\
      &=& \dim A(1-\iota(e)),
  \end{eqnarray*}
  where the four equalities follow, respectively, from the assumption
  that $\iota$ is non-degenerate, from the assumption that $e$ is an idempotent,
  from Lemma~\ref{lem:leftrightrk}, and from the assumption that $\iota$ is an
  anti-automorphism. The claimed equality follows.
\end{proof}

We now prove the main result of the section.

\begin{prop}\label{prop:involidem}
  Let~$A$ be a semisimple $\Q$-algebra with an anisotropic involution~$\iota$.
  Then for every idempotent~$e\in A$, we have~$\regmod{A} = Ae+
  A(1-\iota(e))$.
\end{prop}
\begin{proof}
  By Lemma~\ref{lem:iotaperp} we have~$(Ae)^\perp = A(1-\iota(e))$. Since $\iota$
  is anisotropic, we have~$A = Ae+(Ae)^\perp$, giving the result.
\end{proof}

\begin{example}\label{ex:nonisotropic}
  Proposition~\ref{prop:involidem} is false without the anisotropy assumption,
  even if the algebra is simple. For instance, the split quaternion algebra $A =
  \Mat_2(\Q)$, the involution
      \[
        \begin{pmatrix}a & b \\ c & d\end{pmatrix}\mapsto \begin{pmatrix}d & -b \\
      -c & a\end{pmatrix},
      \]
      and the idempotent $e=\left(\begin{smallmatrix}1&0\\0&0\end{smallmatrix}\right)$
      provide a counter-example. This example shows, in particular, that
      Proposition \ref{prop:involidem} is not a formal consequence of the
      fact that $Ae$ is isomorphic to $A\iota(e)$ as $A$-modules.
\end{example}

\begin{definition}\label{def:invol}
  Let $G$ be a finite group. Recall that the group algebra~$\Q[G]$ is a
  semisimple $\Q$-algebra. Define an involution $x\mapsto x^*$ on $\grpalg$
by setting $g^*= g^{-1}$ for all $g\in G$, and extending $\Q$-linearly.
\end{definition}

\begin{corollary}\label{cor:goodidempotent}
Let $G$ be a finite group. Then for
every idempotent $e\in \grpalg$, we have $\regmodg=\grpalg e+ \grpalg
(1-e^*)$.
\end{corollary}
\begin{proof}
If $a=\sum_{g\in G}a_gg$ is an arbitrary element of $\grpalg$, then the
coefficient of the identity $1\in G$ in $aa^*$ is $\sum_{g\in G}a_g^2$.
It therefore follows from Example~\ref{ex:trace} that the involution $x\mapsto x^*$
is anisotropic. The result follows from Proposition \ref{prop:involidem},
applied to $A=\grpalg$ and $\iota=(x\mapsto x^*)$.
\end{proof}

\section{Proof of the main theorem}\label{sec:geom}
In this section, we prove Theorem \ref{thm:bodymain}, which is a strengthening
of Theorem \ref{thm:intromain} from the introduction.
Let $G$ be a finite group. Our proof will proceed
by a sequence of Dehn surgeries on a $3$-manifold with a $G$-action.

\begin{defns}
\begin{enumerate}[leftmargin=*]
\item  Let~$\origM$ be a manifold with an action of~$G$. We say that a subset
  $C\subseteq \origM$ is \emph{\good}~if for every~$g\in G\smallsetminus\{1\}$,
  the intersection $C\cap gC$ is empty,
  equivalently if the restriction to~$C$ of the covering map~$\origM\to \origM/G$ is
  injective.
\item
Below, the manifolds $S^1$ and $\partial D^2$ are understood to be
oriented. If $\origM$ is a $3$-manifold with a $G$-action, and
$\emb\colon S^1\times D^2\to \origM$ is an embedding with a \good~image,
let $\drilledM{\origM}{\emb}=\origM\smallsetminus G\act\emb(\text{interior of }S^1\times D^2)$.
Let $\killed$, $\compl$ be simple closed curves on $\emb(S^1\times \partial D^2)$
whose classes in the fundamental group $\pi_1(\emb(S^1\times \partial D^2))$ together
generate that fundamental group. Then the result of \emph{$G$-equivariant surgery on $\origM$
along $\emb$, $\killed$} is the manifold $\filledM{\origM}{\emb,\killed}$ defined as
$\drilledM{\origM}{\emb} \sqcup \bigsqcup_{g\in G}g(S^1\times D^2)$, where
each $g(S^1\times D^2)$ is a copy of $S^1\times D^2$, with the obvious
$G$-action on the disjoint union, modulo the equivalence relation that identifies
the boundary of $\drilledM{\origM}{\emb}$ with that of 
$\bigsqcup_{g\in G}g(S^1\times D^2)$ by identifying, for all $g\in G$,
the curve $g\killed$ with the simple closed curve
$g(\{1\}\times \partial D^2)\subset g(S^1\times \partial D^2)$,
and $g\compl$ with the simple closed curve
$g(S^1\times \{1\})\subset g(S^1\times \partial D^2)$. The diffeomorphism
class of $\filledM{\origM}{\emb,\killed}$ does not depend on the choice of $\compl$.
\end{enumerate}
\end{defns}

If $\origM$ is a $3$-manifold, then we denote the intersection pairing
$H_2(M,\Z)\times H_1(M,\Z)\to \Z$ by $(x,y)\mapsto x\cdot y$, and we use the same
notation for $\Q$-coefficients in place of $\Z$.

We will make repeated use of the following variant of \cite[Lemma 5.6]{KM}.

\begin{lemma}\label{lem:KM}
Let $\origM$ be a $3$-manifold, let
$\emb\colon S^1\times D^2\to \origM$ be an embedding with a \good~image, and
let $\killed$ be a simple closed curve on $\emb(S^1\times \partial D^2)$. Then the row and the
column in the diagram
$$
\xymatrix{
& & H_{2}(\filledM{\origM}{\emb,\killed},\Z)\ar[d]^{\dotmu} & & \\
& & \regmod{\Z[G]}\ar[d]^\epsilon & \\
H_{2}(\origM,\Z)\ar[r]^-{\dotlambda} & \regmod{\Z[G]}\ar[r]^-\delta& H_1(\drilledM{\origM}{\emb},\Z)\ar[r]^-{i_*}\ar[d]^{j_*} & H_1(\origM,\Z)\ar[r] & 0\\
& & H_1(\filledM{\origM}{\emb,\killed},\Z)\ar[d] & & \\
& & 0 & & \\
}
$$
of $\Z[G]$-modules are exact, where the maps are defined as follows:
\begin{itemize}[leftmargin=*]
\item $i_*$ and $j_*$
are induced by the canonical injections of $\drilledM{\origM}{\emb}$ into
$\origM$, respectively $\filledM{\origM}{\emb,\killed}$;
\item $\delta$ sends $1$ to the class of $\emb(\{1\}\times \partial D^2)$, and
$\epsilon$ sends $1$ to the class of $\killed$;
\item $\dotlambda=(x\mapsto \sum_{g\in G}(x\cdot g\lambda)g)$, where $\lambda$ is
the curve $\emb(S^1\times\{0\})\subset M$,
and $\dotmu=(x\mapsto \sum_{g\in G}(x\cdot g\mu)g)$, where $\mu$ is the
curve $S^1\times \{0\}\subset S^1\times D^2\subset \filledM{\origM}{\emb,\killed}$.
\end{itemize}
\end{lemma}
\begin{proof}
The proof is identical to that of \cite[Lemma 5.6]{KM}.
\end{proof}
%The strategy of proof of Theorem \ref{thm:bodymain} will be the following.
%\marginpar{Is this helpful?}
%We can start with an arbitrary $3$-manifold with a $G$-action.
%\begin{enumerate}[leftmargin=*]
%\item
%By Lemma \ref{lem:addQG} below, given a $3$-manifold $M$ with a $G$-action
%we can find a $3$-manifold $M'$ with a $G$-action
%such that $H_1(M',\Q)\cong H_1(M,\Q)\oplus \regmodg$.
%\item
%By Proposition \ref{prop:removeeQG},
%given a $3$-manifold $M$ with a $G$-action such that $H_1(M,\Q)\cong \regmodg\oplus U$
%for a $\grpalg$-module $U$, and a $\grpalg$-submodule $V$ of $\regmodg$, we can
%find a $3$-manifold $M'$ with a $G$-action such that $H_1(M',\Q)\cong V\oplus U$.
%\end{enumerate}
%As we will explain in the proof of Theorem \ref{thm:bodymain}, combining these
%steps, we can realise any given $\grpalg$-module as $H_1(N,\Q)$ for a
%$3$-manifold $N$ with a $G$-action, and hyperbolicity can be ensured at the
%end with a final $G$-equivariant Dehn surgery that does not change the homology.
Below we will also use the notation of Lemma \ref{lem:KM} for homology with
$\Q$-coefficients. The two basic Dehn surgeries that we
will use in the proof of Theorem
\ref{thm:bodymain} are described in Lemma \ref{lem:addQG} and Proposition \ref{prop:removeeQG}.

\begin{lemma}\label{lem:addQG}
Let $M$ be a $3$-manifold with a free $G$-action. Then there exists a
$3$-manifold $M'$ with a free $G$-action such that we have an isomorphism
$H_1(M',\Q)\cong H_1(M,\Q)\oplus \regmodg$ of $\grpalg$-modules.
\end{lemma}
\begin{proof}
Let $\emb\colon S^1\times D^2\to M$ be a \good~embedding that sends
$S^1\times \{1\}$ to a null-homotopic
simple closed curve $\killed$ in $M\smallsetminus G\cdot \emb(\text{interior of }S^1\times D^2)$,
and let $M'=M(\emb,\killed)$.
Then, $\lambda$ is also null-homotopic in $M$, so the map $\dotlambda$
of Lemma \ref{lem:KM} is the zero map, so that the map~$\delta$ of the lemma
is injective; and also, since $\gamma$ is null-homotopic, the map $\epsilon$ of Lemma \ref{lem:KM}
is the zero map, so that the map $j_*$ of the lemma is injective. By
Lemma~\ref{lem:KM}, the manifold~$M'$ has the required property.
\end{proof}

Recall from Definition \ref{def:invol} the involution on $\grpalg$ given
by $g\mapsto g^*=g^{-1}$ for all $g\in G$.

\begin{lemma}\label{lem:genericity}
Let $G$ be a finite group, let $e\in \grpalg$ be an idempotent, and let
$y\in \grpalg$ be arbitrary. Let $A=\Q[G]e$, and
for $s\in \Q$ let $B(s)=\{b\in \grpalg:b(1+sy)\in \Q[G](1-e^*)\}$.
Then for all but finitely many $s\in \Q$, we have $A\cap B(s)=\{0\}$.
\end{lemma}
\begin{proof}
For all but finitely many $s\in \Q$, the element $1+sy$ is invertible,
since the multiplication-by-$y$ map on $\grpalg$ has only finitely many
eigenvalues. This implies that for all but finitely many $s\in \Q$, the $\Q$-vector
subspace $B(s)$ of $\grpalg$
has dimension $\dim\grpalg(1-e^*)=\dim\grpalg(1-e)=\dim\grpalg-\dim A$.
We deduce that for all but finitely many $s\in \Q$, the condition
$A \cap B(s)= \{0\}$ is equivalent to $\dim(A + B(s)) = \dim\grpalg$,
which is equivalent to the non-vanishing of a determinant that is a polynomial
in $s$. Since $B(0)=\grpalg(1-e^*)$, Corollary \ref{cor:goodidempotent}
implies that $A+B(0)=\grpalg$, so the above determinant is not identically $0$,
so has only finitely many roots, as claimed.
\end{proof}

\begin{proposition}\label{prop:removeeQG}
Let $P$ be a submodule of $\regmodg$, let $U$ be a $\grpalg$-module, and
let $M$ be a $3$-manifold with a free $G$-action such that there is an isomorphism
$H_1(M,\Q)\cong \regmodg\oplus U$ of $\grpalg$-modules. Then there exists
a $3$-manifold $M'$ with a free $G$-action such that there is an isomorphism
$H_1(M',\Q)\cong P\oplus U$ of $\grpalg$-modules.
\end{proposition}
\begin{proof}
Let $l\in H_1(M,\Z)$ be such that $\grpalg l$ is the direct summand
$\grpalg l\cong \regmodg$ of $H_1(M,\Q)$,
let $e\in \grpalg$ be an idempotent such that we have an isomorphism $\grpalg e\cong P$
of $\grpalg$-modules, let $d\in \Z_{>0}$ be such that $de\in \intgrpalg$,
let $\lambda$ be a simple
closed curve in $M$ representing the class $[\lambda]=d(1-e)l\in H_1(M,\Z)$, and let
$\emb\colon S^1\times D^2\to M$ be a \good~embedding such that $\emb(S^1\times\{0\})=\lambda$.

First, we claim that, with the above choices, the kernel of the map $\delta$
of Lemma \ref{lem:KM} is $\Q[G](1-e^*)$, so that $H_1(M(\emb),\Q)\cong H_1(M,\Q)\oplus
\Q[G]e^*\cong H_1(M,\Q)\oplus P$, with the summand $P$ being generated by
the class of $\emb(\{1\}\times\partial D^2)$. To prove the claim, let $x\in H_2(M,\Q)$
have intersection number $1$ with $l$, and intersection
number $0$ with $gl$ for all $g\in G\smallsetminus \{1\}$ and with all
classes in $U$. Such an element exists by Poincar\'e duality.
Then it is clear that the image of $f_\lambda$ is generated, as
a $\grpalg$-module, by $f_\lambda(x)$. Write $d(1-e)=\sum_{h\in G}a_hh$, where $a_h\in \Z$
for all $h\in G$.
For every $g\in G$, the intersection number $x\cdot g\lambda$ is the coefficient
in $gd(1-e)$ of the identity, which is $a_{g^{-1}}$, so that
$f_{\lambda}(x) = \sum_{g\in G}(x\cdot g\lambda)g=\sum_{g\in G}a_{g^{-1}}g=d(1-e^*)$,
as claimed.

Let $\meria$ be the class in $H_1(M(\emb),\Q)$ of the simple closed curve
$\emb(\{1\}\times \partial D^2)\subset\emb(S^1\times\partial D^2)$, and let
$\paraa$ be the class of a simple closed curve on $\emb(S^1\times\partial D^2)$
such that $i_*(\paraa) = [\lambda]$.

Our second claim is  that for all
but finitely many values of $q/p$, where~$p$ and $q$ are coprime integers, if
$\killed$ is a simple closed curve on $\emb(S^1\times\partial D^2)$ whose class
$[\killed]$ in $H_1(M(\emb),\Q)$ is $p\meria + q\paraa$, then the map $\epsilon$
of Lemma~\ref{lem:KM} is injective. Indeed, let $p$ and $q$ be coprime integers,
let $\gamma$ be as just described, and suppose that $a\in \regmodg$ is such
that $\epsilon(a)=a[\killed] = 0$. Let $C$ be a $\grpalg$-submodule of $H_1(M(\emb),\Q)$
such that there is a direct sum decomposition of $\grpalg$-modules
\begin{eqnarray}\label{eq:directsum}
H_1(M(\emb),\Q)=\delta(\regmodg)\oplus C=\grpalg \meria \oplus C.
\end{eqnarray}
Then we may write
$\paraa = y\meria+c$, where $y\in \grpalg$ and $c\in C$, so that
$a[\killed] = (ap+aqy)\meria + aqc$. By the direct sum decomposition (\ref{eq:directsum}),
%we have $a[\killed]=0$ if and only if 
the assumption that~$a[\killed]=0$ is equivalent to
\begin{enumerate}[leftmargin=*,label={\upshape(\roman*)}]
\item\label{item:first} $(ap+aqy)\meria=0$ and
\item\label{item:second} $aqc=0$.
\end{enumerate}

Since $\grpalg\meria$ is
precisely the kernel of $i_*$,
condition \ref{item:second} is equivalent to
$0=i_*(aqc)=i_*(aq\paraa) = aq[\lambda] = aqd(1-e)l\in H_1(M,\Q)$, which, for $q\neq 0$,
is equivalent to $a$ being contained in the left annihilator $A$ of $1-e$. That
annihilator is equal to $\grpalg e$.
Condition \ref{item:first}, on the other hand, is equivalent to
$0=(ap+aqy)\meria=\delta(ap+aqy)$, in other words to $a$ being contained in
$B(\tfrac{q}{p})=\{b\in \grpalg:b(1+\tfrac{q}{p}y)\in \ker\delta\}$.
Since $\ker \delta=\grpalg(1-e^*)$, Lemma \ref{lem:genericity} implies that for
all but finitely many values of $q/p$, we have $A\cap B(\tfrac{q}{p})=\{0\}$,
which proves the second claim.

It immediately follows from the two claims and Lemma \ref{lem:KM} that for
all but finitely many values of $q/p$, we have
$$
H_1(M(\emb,\killed),\Q)\oplus \regmodg\cong H_1(M(\emb),\Q)\cong H_1(M,\Q)\oplus P,
$$
so that $M'=M(\emb,\killed)$ satisfies the conclusion of the proposition.
\end{proof}

\begin{corollary}\label{cor:add}
Let $M_0$ be a $3$-manifold with a free $G$-action, and let $V$ be a $\grpalg$-module.
Then there exists a $3$-manifold $M$ with a free $G$-action such that
$H_1(M,\Q)\cong H_1(M_0,\Q)\oplus V$.
\end{corollary}
\begin{proof}
Let $V\cong \bigoplus_{i=1}^r P_i$, where each $P_i$ is a submodule of
$\regmodg$. Define the $3$-manifolds $M_i$ and $M_i'$
inductively as follows: supposing that $M_{i-1}$ has been defined, by Lemma
\ref{lem:addQG} there exists a $3$-manifold $M_i'$ with a free $G$-action such
that $H_1(M_i',\Q)\cong H_1(M_{i-1},\Q)\oplus \regmodg$,
and by Proposition~\ref{prop:removeeQG} applied to the manifold $M_i'$
and the $\grpalg$-module $P_i$, there exists a $3$-manifold $M_i$
with a free $G$-action such that $H_1(M_i,\Q)\cong H_1(M_{i-1},\Q)\oplus P_i$.
The manifold $M=M_r$ then satisfies the conclusions of the corollary.
\end{proof}

By a \emph{hyperbolic manifold} we mean a connected oriented smooth
manifold whose interior is equipped with a Riemannian metric with constant
curvature $-1$. We can now deduce the main theorem, which is stronger than
Theorem \ref{thm:intromain} and which reads as follows.
\begin{thm}\label{thm:bodymain}
Let $G$ be a finite group, and let $W$ be a finitely generated $\grpalg$-module.
Then there exists a closed~hyperbolic~$3$-manifold $M'$ with a free $G$-action
such that the $\grpalg$-module $H_1(M',\Q)$ is isomorphic to~$W$.
\end{thm}
\begin{proof}
Let $M_0$ be a $3$-manifold with a free $G$-action. There are many constructions
of such manifolds, see e.g. \cite[\S 1]{CooperLong}. We can apply Corollary
\ref{cor:add} to obtain a $3$-manifold $M_1$ with a free $G$-action such that for
some integer $n\geq 1$, there exists an isomorphism $H_1(M_1,\Q)\cong \regmodg^{\oplus n}\oplus W$
of $\grpalg$-modules. By repeated application of Proposition \ref{prop:removeeQG}
with $P=\{0\}$, we may obtain a $3$-manifold $M_2$ with a free $G$-action
such that there is an isomorphism of $\grpalg$-modules $H_1(M_2,\Q)\cong W$.

We now follow the argument of \cite[Theorem 2.6]{CooperLong} to
obtain a hyperbolic such manifold. Let $p\colon M_2\rar M_2/G$ be the
covering map.
By \cite[Proposition~4.2]{hyperbolise}, the manifold $M_2/G$ contains a
null-homotopic simple closed curve~$k$ such that $(M_2/G)\smallsetminus k$ is a
complete hyperbolic manifold with a single cusp and such that $p^{-1}(k)$ is a
union of $\#G$ simple closed curves that bound disjoint discs in $M_2$
(see also the first paragraph of \cite[Proof of Lemma 4.3]{hyperbolise}).
Let $\emb\colon S^1\times D^2\to M_2$
be an embedding with a \good~image such that $\emb(S^1\times \{0\})$ is one of
these simple closed curves. By \cite[Lemma 4.3]{hyperbolise}, for all but one slope
$\killed$ on $\emb(S^1\times \partial D^2)$, the $G$-equivariant surgery along
$\emb$, $\killed$ on $M_2$ yields a closed manifold $M_2(\emb,\killed)$ with a
free $G$-action, satisfying $H_1(M_2(\emb,\killed),\Q)\cong W$.
By Thurston's hyperbolic Dehn surgery theorem \cite[Theorem 5.8.2]{Thurston},
equivariant surgery for all but finitely many of these slopes results in
a hyperbolic manifold $M'$.
\end{proof}
\begin{remark}
  The last paragraph of the above proof can be replaced by an appeal to
  Theorem A in the very recent preprint \cite{equivHyper}.
\end{remark}

\section{Homology and the structure of fixed point sets}\label{sec:otherdim}
In this section, we first briefly discuss the analogues of the results in Section
\ref{sec:geom} for $G$-actions that are not necessarily free. We will omit
most details, since the proofs are essentially identical to those of Section
\ref{sec:geom}. We then compare these results to the very different situation
of group actions on $2$-dimensional manifolds.

The proof of Lemma \ref{lem:addQG} goes through in the following
greater generality: we may allow $M$ to have a $G$-stable ``bad region''
$M^{\bad}\subseteq M$ that is allowed to be an orbifold, and in which
non-trivial elements of $G$ are allowed to have fixed points. This set will then
be avoided during the sequence of surgeries. Moreover, the proof of Proposition
\ref{prop:removeeQG} also goes through in that generality, as long as the summand
$\regmodg$ of $H_1(M,\Q)$ is contained in the image of the natural map
$H_1(M\smallsetminus M^{\bad},\Q)\to H_1(M,\Q)$. One therefore deduces
the following generalisations of Corollary \ref{cor:add} and of Theorem
\ref{thm:bodymain}.
In the next two results, let $\cC$ be the category of connected topological
$3$-dimensional orbifolds, possibly with boundary, and let $\cC'$ be the full
subcategory of $\cC$ whose objects are oriented manifolds without boundary.
All group actions will be assumed to be by homeomorphisms.

\begin{theorem}\label{thm:generaladd}
Let $M_0\in \cC$, with an action of a finite group $G$, and let
$V$ be a finitely generated $\grpalg$-module. Let $M_0^{\bad}\subseteq M_0$
be a subset that satisfies the following conditions:
\begin{enumerate}[leftmargin=*, label={(\alph*)}]
\item $M_0^{\bad}$ is $G$-stable;
\item the complement $M_0\smallsetminus M_0^{\bad}$ is in $\cC'$;
\item the group $G$ acts freely on $M_0\smallsetminus M_0^{\bad}$ by orientation
preserving automorphisms.
\end{enumerate}
Then there exists $M\in \cC$
with a $G$-action, and a $G$-stable subset $M^{\bad}\subseteq M$ such that
\begin{enumerate}[leftmargin=*]
\item the complement $M\smallsetminus M^{\bad}$ is in $\cC'$, and
$G$ acts freely by orientation preserving automorphisms on it,
\item there is a $G$-equivariant homeomorphism from
$M_0^{\bad}$ to $M^{\bad}$,
\item there is an isomorphism of $\grpalg$-modules
$H_1(M,\Q)\cong H_1(M_0,\Q)\oplus V$.
\end{enumerate}
\end{theorem}

\begin{theorem}\label{thm:general}
Let $M_0\in \cC$ be such that $H_1(M_0,\Q)$ is finite dimensional over $\Q$,
with an action of a finite group $G$,
and let $W$ be a finitely generated $\grpalg$-module. Let
$M_0^{\bad}\subseteq M_0$ be a subset that satisfies the following conditions:
\begin{enumerate}[leftmargin=*, label={(\alph*)}]
\item $M_0^{\bad}$ is $G$-stable;
\item the complement $M_0\smallsetminus M_0^{\bad}$ is in $\cC'$;
\item the group $G$ acts on $M_0\smallsetminus M_0^{\bad}$ freely by orientation
preserving automorphisms;
\item\label{item:surj} the canonical map
$H_1(M_0\smallsetminus M_0^{\bad},\Q)\rar H_1(M_0,\Q)$ is surjective.
\end{enumerate}
Then there exists $M\in \cC$ with a $G$-action, and a subset
$M^{\bad}\subseteq M$ such that
\begin{enumerate}[leftmargin=*]
\item the complement $M\smallsetminus M^{\bad}$ is in $\cC'$, and
$G$ acts freely by orientation preserving automorphisms on it,
\item there is a $G$-equivariant
homeomorphism from $M_0^{\bad}$ to $M^{\bad}$,
\item there is an isomorphism of $\grpalg$-modules $H_1(M,\Q)\cong W$.
\end{enumerate}
\end{theorem}
\begin{remark}
Condition \ref{item:surj} is automatically satisfied if
$M_0^{\bad}$ is a finite union of at most $1$-dimensional submanifolds,
possibly with boundary. In particular, such an $M_0^{\bad}\subset M_0$ exists
if $M_0$ is an
oriented Riemannian orbifold, and $G$ acts effectively by orientation
preserving isometries.
\end{remark}

Theorem \ref{thm:general} essentially says that one cannot read off the
geometry of the fixed point set in an orientation preserving $G$-action on
a $3$-manifold $M$ from the $\grpalg$-module structure of $H_1(M,\Q)$. We 
now briefly contrast this with the situation in dimension $2$.
We do not claim any originality in what follows, but we have not been able
to find Corollary \ref{cor:dim2}, in particular, stated in the literature.

The discussion will be most conveniently formulated in terms of characters, for
which a general reference is \cite{Isaacs}.
If $G$ is a finite group, and $U$ is a subgroup, we will denote
by $\permchar{U}$ the permutation character corresponding to the $G$-set $G/U$.

\begin{thm}[Artin's Induction Theorem]\label{thm:Artin}
  Let~$G$ be a finite group. The $\Q$-vector space generated by the $\Q$-valued
  characters of $G$ is freely spanned by the permutation characters
  $\permchar{C}$, as $C$ runs over $G$-conjugacy class representatives of
  cyclic subgroups of $G$.
\end{thm}
\begin{proof}
See \cite[Theorem 5.21]{Isaacs},
\end{proof}
The following result can be deduced from the Riemann-Hurwitz formula,
and either the Lefschetz trace formula or Artin's Induction Theorem.
\begin{proposition}\label{prop:2dim}
Let $M$ be a \decent~surface, let $G$ be a group of orientation
preserving automorphisms of $M$, and let $\tau$ denote the genus of~$M/G$.
Let $\cS$ be a full set of $G$-orbit representatives of the ramification points
of the covering $M\rar M/G$, and for each $P\in \cS$, let $S_P$ be the stabiliser
of~$P$ in $G$. Let $\chi$ be the character
corresponding to the $G$-module $H_1(M,\Q)$. Then we have
$$
\chi = 2\permchar{G} + (2\tau - 2 + \#\cS)\permchar{\{1\}} - \sum_{P\in \cS} \permchar{S_P}.
$$
\end{proposition}
\begin{proof}
See \cite[Proposition 2]{Broughton}.
\end{proof}
It follows that, in the situation of Proposition \ref{prop:2dim}, the structure
of the ramification set of the covering $M\rar M/G$ can be read off from
the $\grpalg$-module structure of $H_1(M,\Q)$ in the following precise sense.
\begin{corollary}\label{cor:dim2}
Let $M$ and $M'$ be \decent~surfaces with an action of a finite group
$G$ by orientation preserving automorphisms. If~$P$ is a point on $M$ or $M'$,
let $S_P$ be its stabiliser in $G$. Suppose that the
$\grpalg$-modules $H_1(M,\Q)$ and $H_1(M',\Q)$ are isomorphic.
Then there exists a bijection $\beta$ between the ramification points of the cover
$M\rar M/G$, and those of the cover $M'\rar M'/G$ such that for all ramification
points $P\in M$, we have $S_P=S_{\beta(P)}\leq G$, so that, in particular,
$\beta$ preserves ramification indices.
\end{corollary}
\begin{proof}
Let $\cS$ and $\cS'$ be full sets of $G$-orbit representatives of the
ramification points of $M\rar M/G$, respectively of $M'\rar M'/G$, and let
$\tau$ and $\tau'$ be the genera of $M/G$, respectively of $M'/G$. By Proposition
\ref{prop:2dim}, there is an equality of characters
$$
(2\tau - 2 + \#\cS)\permchar{\{1\}} - \sum_{P\in \cS} \permchar{S_P}=
(2\tau' - 2 + \#\cS')\permchar{\{1\}} - \sum_{P'\in \cS'} \permchar{S_{P'}}.
$$
Since none of the stabilisers $S_P$ and $S_{P'}$ are trivial, and since they
are all cyclic, it follows from Artin's Induction Theorem that there exists
a bijection~$\alpha$ from $\cS$ to $\cS'$ such that for all $P\in \cS$, we have
$\permchar{S_P}=\permchar{S_{\alpha(P)}}$. This condition on the permutation
characters is equivalent to $S_P$ being conjugate to $S_{\alpha(P)}$ in $G$.
Since for every $P\in \cS$, the set of stabilisers of the points in the
$G$-orbit of $P$ is a single conjugacy class of subgroups, the result follows.
\end{proof}

\section{Application to isospectral manifolds}\label{sec:isospectral}
In this section we deduce Theorem \ref{thm:introsiso} from Theorem
\ref{thm:intromain}. Our proof relies on Sunada's group theoretic construction
of isospectral manifolds \cite{Sunada}, and on the formalism of regulator
constants, as introduced by Dokchitser--Dokchitser, see e.g. \cite{tamroot}.

\subsection{Sunada's construction and the Cheeger--M\"uller theorem}
If $p$ is a prime number, we will write $\Z_{(p)}$ for the localisation
of~$\Z$ at $p$, i.e. the subring $\{\frac ab\colon p\nmid b\}$ of $\Q$.
In this subsection, $R$ will be either $\Q$ or $\Z_{(p)}$, where~$p$
is a prime number.
\begin{definition}
Let $G$ be a finite group. An
\emph{$R[G]$-relation} is a formal linear combination
$\sum_i U_i - \sum_j U_j'$ of subgroups of $G$ with the property that there
is an isomorphism of $R[G]$-modules
$$
\bigoplus_i R[G/U_i] \cong \bigoplus_j R[G/U_j'].
$$
\end{definition}

The following lemmas are routine, and we leave the proofs to the reader.
\begin{lemma}
If $G$ is a finite group, $N$ is a normal subgroup, and $\Theta=\sum_i U_i - \sum_j U_j'$
is an $R[G]$-relation, then
$\Def_{G/N}\Theta=\sum_i NU_i/N - \sum_j NU_j'/N$ is an $R[G/N]$-relation.
\end{lemma}
\begin{lemma}\label{lem:GassmannProduct}
Let $G$ and $\tilde{G}$ be finite groups, let $U-U'$ be an $R[G]$-relation,
and let $\tilde{U}-\tilde{U}'$ be an $R[\tilde{G}]$-relation. Then
$U\times \tilde{U} - U' \times \tilde{U}'$ is a $R[G\times \tilde{G}]$-relation.
\end{lemma}

In \cite{Sunada} Sunada shows that $\Q[G]$-relations give rise to strongly
isospectral manifolds, as follows.
\begin{theorem}[Sunada, \cite{Sunada}]
Let $G$ be a finite group, let $X\rar Y$ be a $G$-covering of Riemannian manifolds,
and let $U-U'$ be a $\Q[G]$-relation. Then the intermediate coverings $X/U$ and
$X/U'$ are strongly isospectral.
\end{theorem}

It follows from the Cheeger--M\"uller Theorem \cite{Cheeger,Mueller1,Mueller2}
that if $M$ and $M'$ are strongly isospectral Riemannian $3$-manifolds,
then
\begin{eqnarray}\label{eq:CheeegerMuller}
\frac{\#H_1(M,\Z)_{\tors}}{\#H_1(M',\Z)_{\tors}}=\frac{\Reg_1(M)^2}{\Reg_1(M')^2},
\end{eqnarray}
where $\Reg_1(M)$ is the covolume of the
lattice~$H_1(M,\Z)/H_1(M,\Z)_{\tors}$ in the vector space $H_1(M,\bR)$ with respect to
a certain canonical inner product, and similarly for $M'$ -- see \cite{us} for details.

\subsection{Regulator constants}
When $M$ and $M'$ arise from a $G$-covering $X\rar Y$ via Sunada's construction,
we relate in \cite{us} the regulator quotient of equation (\ref{eq:CheeegerMuller})
to a certain representation theoretic invariant of $H_1(X,\Q)$, called a regulator
constant. We briefly recall the definition and some of the properties of
this invariant. On the first reading, the definition may be skipped, since only
the properties that we list below will be needed for the rest of the section.

\begin{definition}\label{def:regconst}
Let $G$ be a finite group, let $\Theta=\sum_i U_i - \sum_j U_j'$ be a
$\Q[G]$-relation, and let
$W$ be a finitely generated $\grpalg$-module. Let $\langle\;,\;\rangle$
be a non-degenerate $G$-invariant $\Q$-bilinear pairing on $W$ with values in $\Q$.
The \emph{regulator constant} of $W$ with respect to $\Theta$ is defined as
$$
\cC_{\Theta}(W) = \frac{\prod_i \det\left(\frac{1}{\#U_i}\langle\;,\;
\rangle|W^{U_i}\right)}{\prod_j \det\left(\frac{1}{\#U_j'}\langle\;,\;
\rangle|W^{U_j'}\right)}\in \Q^{\times}/(\Q^{\times})^2.
$$
Here, each determinant is evaluated with respect to an arbitrary basis
of the respective fixed space, and is therefore only well-defined
modulo $(\Q^{\times})^2$.
\end{definition}
\begin{remark}
Let $G$ be a finite group, and let $W$ be a finitely generated $\grpalg$-module.
Choosing a pairing as in Definition \ref{def:regconst} is equivalent to choosing
an isomorphism of $\grpalg$-modules between $W$ and its $\Q$-linear dual.
Since finitely generated $\grpalg$-modules are self-dual, such a pairing always
exists.
\end{remark}
\begin{theorem}\label{thm:indep}
The value of $\cC_{\Theta}(W)$ is independent of the pairing $\langle\;,\;\rangle$.
\end{theorem}
\begin{proof}
See \cite[Theorem 2.17]{tamroot}.
\end{proof}
Theorem \ref{thm:indep} justifies the notation $\cC_{\Theta}(W)$, which makes
no reference to the pairing.
\begin{example}\label{ex:GL}
Let $p$ be an odd prime number, and let $G_p=\GL_2(\F_p)$ be the group of
invertible $2\times 2$ matrices over the finite field with $p$ elements.
Consider the following subgroups of $G_p$:
$$
B_p=\begin{pmatrix}\F_p^\times &\F_p\\0& \F_p^\times\end{pmatrix},\;\;
U_p=\begin{pmatrix}(\F_p^\times)^2 &\F_p\\0& \F_p^\times\end{pmatrix},\;\;
U_p'=\begin{pmatrix}\F_p^\times &\F_p\\0& (\F_p^\times)^2\end{pmatrix}.
$$
The permutation module $\Q[G_p/U_p]$ decomposes as a direct sum $\Q[G_p/U_p]\cong
\Q[G_p/B_p]\oplus I_p$,
where $I_p$ is a simple $\Q[G_p]$-module of dimension $(p+1)$ over~$\Q$.
Moreover, the formal linear combination $\Theta=U_p-U_p'$ is a $\Q[G_p]$-relation,
and for every prime number $q\neq p$, it is a $\Z_{(q)}[G_p]$-relation.
In \cite[Proposition~4.2]{us} we showed that
$\cC_{\Theta}(I_p)\equiv p\pmod{(\Q^\times)^2}$.
\end{example}

\begin{example}\label{ex:2-Gassmann}
Let $G_2$ be the affine linear group over $\Z/8\Z$, i.e. the group of linear
transformations $T_{a,b}\colon x\mapsto ax+b$ of $\Z/8\Z$, where $a\in (\Z/8\Z)^{\times}$
and~$b\in \Z/8\Z$. Consider the following subgroups of $G_2$:
\begin{eqnarray*}
U_2 & = & \langle T_{a,0}\colon a \in (\Z/8\Z)^\times\rangle,\\
U_2' & = & \langle T_{3,4}, T_{-1,0} \rangle,\\
B_2 & = & \langle T_{3,4}, T_{a,0}\colon a \in (\Z/8\Z)^\times\rangle.
\end{eqnarray*}
The group $G_2$ is isomorphic to the semidirect product $\Z/8\Z\rtimes (\Z/8\Z)^\times$;
the subgroups $U_2$ and $U_2'$ are both isomorphic to $C_2\times C_2$;
$\Theta=U_2-U_2'$ is a $\Q[G_2]$-relation, and for every odd prime number $q$, it
is a $\Z_{(q)}[G_2]$-relation.
Moreover, $\Q[G_2/U_2]$ decomposes as a direct sum $\Q[G_2/U_2]\cong \Q[G_2/B_2]\oplus I_2$,
where $I_2$ is a simple $\Q[G_2]$-module of dimension 4 over $\Q$, and
one can show by a direct computation that $\cC_{\Theta}(I_2)\equiv 2\pmod{(\Q^\times)^2}$.
\end{example}\noindent
Regulator constants satisfy the following properties:
\begin{enumerate}[label={(Reg \arabic*)}]
\item\label{item:Reg1} if $G$ is a finite group, $N$ is a normal subgroup,
$\Theta$ is a $\Q[G]$-relation,
$W$ is a $\Q[G/N]$-module, and $\Inf_{G/N}W$ is the lift of $W$
to a $\grpalg$-module, then $\cC_{\Theta}(\Inf_{G/N}W)=\cC_{\Def_{G/N}\Theta}(W)$;
\item\label{item:Reg2} if $G$ is a finite group, $\Theta$ is a $\Q[G]$-relation,
and $W_1$, $W_2$ are $\grpalg$-modules, then $\cC_{\Theta}(W_1\oplus W_2)=
\cC_{\Theta}(W_1)\cdot\cC_{\Theta}(W_2)$.
\end{enumerate}

\begin{lem}\label{lem:biggroup}
  Let~$\cP$ be a finite set of prime numbers. Then there exist a finite
  group~$G$, a~$\Q[G]$-relation~$\Theta=U-U'$, and a~$\Q[G]$-module~$W$, such
  that
  \begin{enumerate}[leftmargin=*]
    \item\label{item:regcst} we have
      \[
        \cC_\Theta(W) \equiv \prod_{p\in\cP}p \pmod{(\Q^\times)^2};
      \]
    \item\label{item:qrel} for all prime numbers $q\not\in \cP$, the relation~$\Theta$ is
      a~$\Z_{(q)}[G]$-relation.
  \end{enumerate}
\end{lem}
\begin{proof}
Let $G=\prod_{p\in\cP} G_p$, where $G_p$ is as in Example~\ref{ex:GL}
when~$p$ is odd, and as in Example~\ref{ex:2-Gassmann} when $p=2$.
For each~$p\in\cP$, let $N_p$ denote the kernel of the projection map~$G\to
G_p$, so that the quotient $G/N_p$ is isomorphic to~$G_{p}$.

We may lift the module $I_{p}$ of Example~\ref{ex:GL}, respectively \ref{ex:2-Gassmann} from $G/N_p$ to a $\grpalg$-module $\Inf_{G/N_p}I_{p}$. Let $W$ be
the direct sum of $\grpalg$-modules $W=\bigoplus_{p\in\cP} \Inf_{G/N_p}I_{p}$.
Let $U=\prod_{p\in\cP} U_{p}\leq G$, where the subgroups~$U_{p}\leq G_{p}$ are
as in Example~\ref{ex:GL}, respectively~\ref{ex:2-Gassmann},
and define $U'$ analogously. So for every $p\in \cP$, the image
of $U$ under the quotient map $G\rar G/N_p$ is $U_{p}$, and
the image of $U'$ is $U_{p}'$.

By Lemma \ref{lem:GassmannProduct},
the formal linear combination $\Theta=U-U'$ is a $\Q[G]$-relation,
and for every prime number $q\not\in \cP$, it is also a $\Z_{(q)}[G]$-relation.
This proves the second part of the lemma.

By property~\ref{item:Reg2}, property~\ref{item:Reg1}, and
Examples~\ref{ex:GL} and~\ref{ex:2-Gassmann}, in that order, we have
\begin{eqnarray*}
  \cC_{\Theta}(W) & \equiv & \prod_{p\in\cP} \cC_{\Theta}(\Inf_{G/N_p}I_{p})
                  \equiv \prod_{p\in\cP} \cC_{\Def_{G/N_p}\Theta}(I_{p})\\
                  & \equiv & \prod_{p\in\cP} p \pmod{(\Q^\times)^2},
\end{eqnarray*}
%$$
%  \cC_{\Theta}(W) \equiv \prod_{p\in\cP} \cC_{\Theta}(\Inf_{G/N_p}I_{p})
%                  \equiv \prod_{p\in\cP} \cC_{\Def_{G/N_p}}(I_{p})
%                  \equiv \prod_{p\in\cP} p \pmod{(\Q^\times)^2},
%$$
which proves the first part of the lemma.
\end{proof}

\subsection{Isospectral manifolds}
The following two results are the crucial ingredients that will allow us to deduce Theorem
\ref{thm:introsiso} from Theorem \ref{thm:intromain}.
\begin{proposition}\label{prop:regquo}
Let $G$ be a finite group, let $X\to Y$ be a~$G$-covering of Riemannian
manifolds, and let~$\Theta=U - U'$ be a $\Q[G]$-relation. Then
$\frac{\Reg_1(X/U)^2}{\Reg_1(X/U')^2}\in \Q^{\times}$, and we~have
$$
\frac{\Reg_1(X/U)^2}{\Reg_1(X/U')^2}\equiv \cC_{\Theta}(H_1(X,\Q)) \pmod{(\Q^\times)^2}.
$$
\end{proposition}
\begin{proof}
This is a special case of \cite[Corollary 3.12]{us}.
\end{proof}
\begin{proposition}\label{prop:Z_qrel}
Let $G$ be a finite group, let $X\rar Y$ be a~$G$-covering of Riemannian
manifolds, let $q$ be a prime number, and let $\Theta=U - U'$ be a
$\Z_{(q)}[G]$-relation. Then we have
$$
H_1(X/U,\Z)[q^\infty]\cong H_1(X/U',\Z)[q^\infty].
$$
\end{proposition}
\begin{proof}
This is a special case of \cite[Theorem 3.5]{us}.
\end{proof}

We can now prove Theorem \ref{thm:introsiso}. We recall the statement.
\begin{theorem}
  Let $\cP$ be a finite set of prime numbers. Then
  there exist \decent~$3$-manifolds $M$ and $M'$ that are strongly isospectral with
  respect to hyperbolic metrics and such that
  \begin{enumerate}[leftmargin=*]
  \item for all $p\in \cP$ we have
  $$
  \#H_1(M,\Z)[p^\infty]\neq \#H_1(M',\Z)[p^\infty];
  $$
  \item for all prime numbers $q\not\in \cP$ we have an isomorphism of Abelian groups
  $$
  H_1(M,\Z)[q^\infty]\cong H_1(M',\Z)[q^\infty].
  $$
\end{enumerate}
\end{theorem}
\begin{proof}
  Let~$G$, $U$, $U'$, $\Theta$, and~$W$ be as in Lemma~\ref{lem:biggroup} applied to the set~$\cP$.
  By Theorem \ref{thm:intromain}, there exists a closed hyperbolic
  $3$-manifold~$X$ with a free $G$-action such that there is an isomorphism
  of $\grpalg$-modules $H_1(X,\Q)\cong W$. Let $M=X/U$ and $M'=X/U'$. The
  second part of the theorem immediately follows from
  Lemma~\ref{lem:biggroup}~(\ref{item:qrel}) and
  Proposition~\ref{prop:Z_qrel}.

  To prove the first part, we
  invoke equation (\ref{eq:CheeegerMuller}), Proposition \ref{prop:regquo},
  and Lemma~\ref{lem:biggroup}~(\ref{item:regcst}), in that order, to conclude that
  \begin{eqnarray*}
  \frac{\#H_1(M,\Z)_{\tors}}{\#H_1(M',\Z)_{\tors}} & = &  \frac{\Reg_1(M)^2}{\Reg_1(M')^2}\\
  & \equiv & \cC_{\Theta}(H_1(X,\Q)) = \cC_{\Theta}(W)\\
  & \equiv & \prod_i p_i \pmod{(\Q^\times)^2},
  \end{eqnarray*}

  which completes the proof.
\end{proof}

\end{document}